\documentclass[11pt, a4paper]{amsart}
\swapnumbers
         
\theoremstyle{plain}
  \begingroup
        \newtheorem{theorem}[equation]{Theorem}
        \newtheorem{lemma}[equation]{Lemma}
  \endgroup

\newcommand{\mr}[1]{\buildrel {#1} \over \longrightarrow}
\newcommand{\cqd}{\hfill$\Box$}

\begin{document}

{ \hfill A SIMPLE PROOF OF MCNAUGHTON THEOREM  \hfill}

\vspace{1ex}

{ \hfill \it{Eduardo J. Dubuc and Yuri Poveda}  \hfill}
 
\vspace{5ex}
   
In this note we exhibit a very simple proof of McNaughton Theorem,
almost right out of the definitions, and
at the same time we observe  that this theorem does not depend of
Chang's completeness theorem.

We recall first a proof of the Chinese theorem for MV-algebras:
%
%
\begin{lemma} Given any MV-algebra $A$ and $a,b,c \in A$, 
$$ 
if \;\; a \leq b \oplus c, \;\;\;  then \;\;\; a \ominus b \leq c 
$$
\end{lemma}
\begin{proof} $a \ominus b \leq (b \oplus c)
\ominus b = (\neg b) \wedge c \leq c$
\end{proof}
\begin{lemma} Given any MV-algebra $A$ and a
  pair of ideals $I_1, I_2$ of $A$, we have:  

If $a_1, a_2 \in
A$ are such that $a_1 \equiv a_2 \left(I_1,\, I_2 \right)$, then there is $a \in A$ such that
$a \equiv a_1 \left(I_1 \right)$, and $a \equiv a_2 \left(I_2 \right)$.
\end{lemma}
\begin{proof} We know that $a_1 \ominus a_2
\,\text{and}\, \hspace{0.1cm} a_2 \ominus a_1 \in (I_1,\, I_2)$.
Thus there are \mbox{$c_1, d_1 \in I_1$} and $c_2, d_2 \in I_2$ such that
$a_1 \ominus a_2 \leq c_1 \oplus c_2$ and $ a_2 \ominus a_1 \leq d_1
\oplus d_2$. From the lemma above we have that $(a_1 \ominus a_2) \ominus
c_1  \leq  c_2, \,\, \text{and} \,\, (a_2 \ominus a_1)
\ominus d_2 \leq d_1$. Then $(a_1 \ominus a_2) \ominus c_1
\in I_2$ and $(a_2 \ominus a_1)
\ominus d_2 \in I_1$\\
Set $$a = (a_1 \ominus a_2 \ominus c_1) \oplus (a_2 \ominus a_1
\ominus d_2) \oplus (a_1 \wedge a_2)$$ 
Then:
$$
[a]_{I_1} = ([a_1]_{I_1} \ominus [a_2]_{I_1}) \oplus ([a_1]_{I_1} \wedge
  [a_2]_{I_1}) = [a_1]_{I_1}\,,
$$

\vspace{1ex}

the second equation holds by \cite[1.6.2.]{COM}. Similarly for $[a]_{I_2} = [a_2]_{I_2}$.
\end{proof}

Taking into account the equation $\bigcap_{i=1}^{n-1}(I_i,\,
I_n) = (\,\bigcap_{i=1}^{n-1}I_i,\, I_n)$, this lemma
easily generalize by induction to a finite number of ideals.  
\begin{theorem}[\cite{FL} 2.6.] \label{chino} 
Given any MV-algebra $A$ and a finite number of ideals  $I_1, \cdots, I_n$
of $A$, we have: 

If $a_1, \cdots, a_n \in A$ are such that $a_i \equiv a_j
(I_i,\, I_j)$ for $i,j \in \{1, \cdots, n \}$,  then there exists
$a \in A$ such that $a \equiv a_i (I_i)$ for $i \in \{1, \cdots,
n\}$. \cqd
\end{theorem}

Let $F_n$ be the free MV-algebra on n generators, that is, the algebra
of terms in $n$ variables. Recall that for any linear
polynomial in $n$ variables with integer coefficients $g$
there is a
  term $a \in F_n$ such that $f_a = g^\sharp$ as
  functions $[0,\,1]^n \mr{} [0,\,1]$, where $g^\sharp = (g \wedge 0)
  \vee 1$, and $f_a$ is the corresponding term function
  (\cite[3.1.9]{COM}).

Recall that a
function $[0,\,1]^n \mr{f} [0,\,1]$ is called a McNaughton function
iff $f$ is continuous with respect to the natural topology of 
$[0,\, 1]^n$, and there are linear polynomials $g_1,\, \ldots \, g_k$ with integer coefficients
such that for each point $x \in [0,\, 1]^n$ there is an index
 $j \in \{1, \, \ldots \, k \}$ with $f(x) = g_j(x)$.

Let $A_n$ be the algebra of term functions. We do not assume Chang's
completeness theorem, so we can use only a surjective morphism $F_n \to
A_n$, not an isomorphism. We will apply theorem \ref{chino} to the
algebra $A_n$.
 
\begin{theorem}[McNaughton]
Given any McNaughton function $f$, there is a term $a \in F_n$
such that $f = f_a$. 
\end{theorem}
\begin{proof}
Let $g_1, \, \ldots \, g_k$ be linear constituents for $f$. Each point 
\mbox{$p \in [0,\, 1]^n$} determines (not univocally) a permutation $\sigma$ of the set
$\{1,\, \ldots \, k\}$ by ordering 
\mbox{$g_{\sigma(1)}^\sharp(p) \leq g_{\sigma(2)}^\sharp(p) \leq \,
    \ldots \, g_{\sigma(k)}^\sharp(p)$.}  Let
$u_\sigma \in \{1,\, \ldots \, k\}$ be an index such $f(p) = g_{u_\sigma}(p) =
g_{u_\sigma}^\sharp(p)$.
Then as in \cite[page 67]{COM}, $f(x) = g_{u_\sigma}(x) = g_{u_\sigma}^\sharp(x)$ for
all $x$ in the set $Z_\sigma = \{x \,|\, g_{\sigma(1)}^\sharp(x) \leq g_{\sigma(2)}^\sharp(x) \leq \,
    \ldots \, g_{\sigma(k)}^\sharp(x)\}$.
Let  $h_i \in A_n$ be term functions such that $h_i = g_i^\sharp$. Clearly 
$Z_\sigma$ is the Zero set of the ideal of term functions  $T_\sigma =
(h_{\sigma(1)} \ominus h_{\sigma(2)}\,,
\,\ldots\,, h_{\sigma(k-1)}\ominus  h_{\sigma(k)})$, $Z_\sigma
= Z(T_\sigma)$.  
 Consider the collection of ideals $T_\sigma$ in
$A_n$, and term functions $h_{u_\sigma} \in A_n$. Then given $\sigma, \; \mu$
 as above, $h_{u_\sigma}(x) = h_{u_\mu}(x) = f(x)$ for all $x$
 in $Z_\sigma \cap Z_\mu = Z(T_\sigma , T_\mu )$. It follows from
 \cite[3.4.8]{COM} \footnote{Although \cite[3.4.8]{COM} in \cite{COM} is stated for the
   free algebra $F_n$, it is actually a result on the algebra of term
   functions  $A_n$.}  that $d(h_{u_\sigma},\,
 h_{u_\mu}) \in (T_\sigma , T_\mu )$, where $d$ is the distance
 operation. By Theorem \ref{chino} above, there is a term function
   $h \in A_n$ such that $h \equiv h_{u_\sigma} (T_\sigma)$. Then for any $x \in [0,\,
     1]^n$, $h(x) = h_{u_\sigma}(x) =  g_{u_\sigma}^\sharp(x) =
 f(x)$. Take any $a \in F_n$ with $h = f_a$. 
\end{proof}


\begin{thebibliography}{99}
\bibitem{COM}Cignoli R., D'Ottaviano I., Mundici
  D., \textsl{Algebraic Foundations of Many-valued Reasoning}, Trends
  in Logic Vol 7, Kluwer Academic Puplishers (2000).

\bibitem{FL} Ferraioli A.R., Lettieri A., \textsl{Representations of
MV-algebras by sheaves}, Math. Log. Quart. 57, No. 1 (2011), p.
27-43.
\end{thebibliography}
\end{document}